\documentclass[11pt]{amsart}
\usepackage{times}
\usepackage{amsfonts}
\usepackage{amsmath}
\usepackage{amsthm}
\usepackage{tikz}

\def\id{\mathop{\rm id}}
\def\N{\mathbb{N}}
\def\Orb{\mathop{\rm Orb}}

\theoremstyle{plain}
\newtheorem{thm}{Theorem}

\newtheorem{exa}[thm]{Example}

\theoremstyle{remark}

\title{Topological Transitivity of Nonautonomous Dynamical Systems}

\author{Michal~M\'{a}lek
}
\email{Michal.Malek@math.slu.cz}
\address{Mathematical Institute in Opava, Silesian University in Opava, Na
Rybn\'{\i}\v{c}ku~1, 746 01 Opava, Czech Republic}
\date{December 28, 2024}
\keywords{Nonautonomous dynamical system;
Transitivity.}
\subjclass[2010]{Primary 37B55, 37B02, 37B05}

\begin{document}

\begin{abstract}
This paper explores the concept of topological transitivity in nonautonomous dynamical systems, which are defined as sequences of continuous maps from a compact metric space to itself.

It investigates various conditions (including intersection of any pair of open sets and existence of a dense orbit) that could be taken as definitions of the topological transitivity of a nonautonomous system, and addresses their relation both in the case of a general compact metric space and in the case where, in addition, the space has no isolated point. 

This provides the necessary basis for further investigation of transitivity of nonautonomous dynamical systems.
\end{abstract}
\maketitle

\section{Introduction}
Let $X$ be a compact metric space, as a \textit{nonautonomous dynamical system} we understand 
a sequence $f_{1,\infty} = f_{1},f_{2},f_{3},\ldots$ of continuous maps from $X$ to $X$. 
For any positive integer $n$ we denote by 
$$
    f_{1,n} = f_{n} \circ f_{n-1} \circ \dots \circ f_{1}
$$
the \textit{$n$-th iteration of $f_{1,\infty}$}, additionally we put $ f_{1,0} = \id$. 

Let $x\in X$, define the \textit{orbit of $x$} putting 
$\Orb_{f_{1,\infty}}(x) = \{x, f_{1,1}(x), f_{1,2}(x),\ldots\}$ and \textit{omega-limits set 
of $x$} as the set of all limit points of the sequence $(f_{1,n} (x))_{n=1}^\infty$, which will 
be denoted by $\omega_{f_{1,n}}(x)$. 

The \textit{n-th preimage $f_{1,n}^{-1}(x)$ of a point $x$} is just preimage under $n$-th iteration
i.e. $(f_{1,n})^{-1}(x)$, similarly we define $f_{1,n}^{-1}(A)$ for a set $A\subset X$.   

If all $f_{n}$ are equal to some $f$ then we get ``classical'' 
\textit{autonomous dynamical system}. In this case we we use $f^{n}$ instead of $f_{1,n}$.

\section{Transitivity of nonautonomous dynamical system}

Topological transitivity of a dynamical system is a key property that shows the ability of a system to transition from any part of space to any other part of space through its dynamics. 
This feature is essential for comprehending the system's 
behavior, stability, and long-term development \cite{Kaw,NT,Xia}. 

Although it is a very important and often studied concept, there is no universally 
accepted definition. For autonomous dynamical system given by continuous 
$f\colon X \rightarrow X$ acting in metric space $X$ the two most commonly 
accepted definitions are TT and DO (see below for definition). 
\begin{itemize}
    \item[(TT)] For every pair of nonempty open sets $U,V\subset X$, there is an $n\in\N$ 
        such that $f^{n}(U) \cap V \neq \emptyset$.
   \item[(DO)] There is an $x \in X$ with a dense orbit.
\end{itemize}

In the case of general metric space these condition are independent. If we consider 
compact metric space then TT implies DO (see \cite{Sil} for details). 
In their famous paper \cite{KS} Kolyada and Snoha presented a list of 15 conditions 
equivalent to TT. If we consider compact metric space without periodic points 
then DO also becomes equivalent to TT \cite{KS}. 

In the presented paper we consider a list of 12 conditions from \cite{KS}, those which 
are clearly reformulable for nonautonomous dynamical systems, and we solve their 
relation both in the case of a compact metric space and in the case when, 
additionally, the space has no isolated point.

From now on, let $f_{1,n}$ be a nonautonomous dynamical system on a metric space $X$.
We consider the following conditions. 
\begin{itemize}
    \item[(TT)] For every pair of nonempty open sets $U,V\subset X$, there is an $n\in\N$ 
        such that $f_{1,n}(U) \cap V \neq \emptyset$.
    \item[(TT$_0$)] For every pair of nonempty open sets $U,V\subset X$, there is an $n\in\N\cup\{0\}$ 
        such that $f_{1,n}(U) \cap V \neq \emptyset$.
    \item[(OP)] For every nonempty open $U\subset X$ the set $\bigcup_{n=1}^{\infty} f_{1,n}(U)$ is dense in $X$.  
    \item[(OP$_{0}$)] For every nonempty open $U\subset X$ the set $\bigcup_{n=0}^{\infty} f_{1,n}(U)$ is dense in $X$.  
    \item[(TT$^{-}$)] For every nonempty open sets $U,V\subset X$, there is an $n\in\N$ 
        such that $f_{1,n}^{-1}(U) \cap V \neq \emptyset$.
    \item[(TT$^{-}_{0}$)] For every nonempty open sets $U,V\subset X$, there is an $n\in\N\cup\{0\}$
        such that $f_{1,n}^{-1}(U) \cap V \neq \emptyset$.
    \item[(OP$^{-}$)] For every nonempty open $U\subset X$ the set $\bigcup_{n=1}^{\infty} f_{1,n}^{-1}(U)$ is dense in $X$.  
    \item[(OP$^{-}_{0}$)] For every nonempty open $U\subset X$ the set $\bigcup_{n=0}^{\infty} f_{1,n}^{-1}(U)$ is dense in $X$.  
    \item[(D$\omega$)] There is a $G_\delta$-dense set $A\subset X$ such that $\omega_{f_{1,\infty}}(x) = X$ for any $x\in A$.
    \item[(DDO)] There is a $G_\delta$-dense set $A\subset X$ such that $\Orb_{f_{1,\infty}}(x)$ is dense  
        for any $x\in A$.
    \item[($\omega$)] There is an $x \in X$ such that $\omega_{f_{1,\infty}}(x) = X$.
    \item[(DO)] There is an $x \in X$ with a dense orbit.
\end{itemize}

\section{Case of compact metric spaces}

In this section, we focus on a non-autonomous dynamical system acting on a general compact metric space $X$. In this setting, we will solve the relations between the above conditions, which could be taken as definitions of transitivity of the non-autonomous system.

A version of the following theorem for autonomous systems can be found, for example, in \cite{dV}.
\begin{thm}[TT$_0$ $\Rightarrow$ DO] \label{TTODO}
    Let $f_{1,\infty}$ be a nonautonomous dynamical system on $X$ possessing TT$_0$\ property 
    then it also has DO. 
\end{thm}
\begin{proof}
    Suppose, on the contrary, that the dynamical system does not have the DO property. 
Therefore, for every point $x\in X$ there exists a nonempty open set $O_{x}$ such that 
$\Orb_{f_{1,\infty}}(x) \cap O_{x} = \emptyset$. 
Denote by $(V_{n})_{n=1}^\infty$ a countable basis of $X$. For each $x$ there exist 
an $n_{x}$ for which $V_{n_{x}} \subset O_{x}$. 

For any $n\in \N$ consider set 
$F_{n} = \{ y \in X;\; \Orb_{f_{1,\infty}}(y) \cap V_{n} = \emptyset\}$.

Each of $F_{n}$ is nowhere dense. To see this consider open $W \subset F_{n}$ by TT$_0$ 
$\Orb_{f_{1,\infty}}(W) \cap V_{n} \neq \emptyset$ contrary to the definition of $F_{n}$. 

But every $x \in X$  lies in some $F_{n}$ which gives us $X = \bigcup_{n=1}^\infty F_{n}$, 
 which is in contradiction with the fact that $X$ is second category.
\end{proof}

\begin{exa}[$\omega \not \Rightarrow$ TT$_0$] \label{ExOmega}
There is a nonautonomous dynamical system $f_{1,\infty}$ on compact space 
having $\omega$ but no TT$_{0}$.
\end{exa}
\begin{proof}
Take $X=\{1,2,3\}$ and a sequence of maps $f_{1,\infty}$ on $X$ cyclically 
repeating three maps as shown on the following picture.
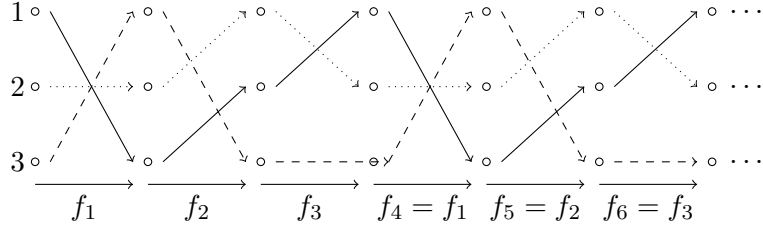
\begin{center}
\begin{tikzpicture}
  \draw (0,0) circle [radius=0.05] node[left] {$3$};
  \draw (0,1) circle [radius=0.05] node[left] {$2$};
  \draw (0,2) circle [radius=0.05] node[left] {$1$};
  \draw (1.5,0) circle [radius=0.05] ;
  \draw (1.5,1) circle [radius=0.05] ;
  \draw (1.5,2) circle [radius=0.05] ;
  \draw (3,0) circle [radius=0.05] ;
  \draw (3,1) circle [radius=0.05] ;
  \draw (3,2) circle [radius=0.05] ;
  \draw (4.5,0) circle [radius=0.05] ;
  \draw (4.5,1) circle [radius=0.05] ;
  \draw (4.5,2) circle [radius=0.05] ;
  \draw (6,0) circle [radius=0.05] ;
  \draw (6,1) circle [radius=0.05] ;
  \draw (6,2) circle [radius=0.05] ;
  \draw (7.5,0) circle [radius=0.05] ;
  \draw (7.5,1) circle [radius=0.05] ;
  \draw (7.5,2) circle [radius=0.05] ;
  \draw (9,0) circle [radius=0.05] node[right] {$\ \ldots$};
  \draw (9,1) circle [radius=0.05] node[right] {$\ \ldots$};
  \draw (9,2) circle [radius=0.05] node[right] {$\ \ldots$};
\draw[->] (0.2,2) -- (1.3,0);
\draw[->] (1.7,0) -- (2.8,1);
\draw[->] (3.2,1) -- (4.3,2);
\draw[->] (4.7,2) -- (5.8,0);
\draw[->] (6.2,0) -- (7.3,1);
\draw[->] (7.7,1) -- (8.8,2);
\draw[->,dotted] (0.2,1) -- (1.3,1);
\draw[->,dotted] (1.7,1) -- (2.8,2);
\draw[->,dotted] (3.2,2) -- (4.3,1);
\draw[->,dotted] (4.7,1) -- (5.8,1);
\draw[->,dotted] (6.2,1) -- (7.3,2);
\draw[->,dotted] (7.7,2) -- (8.8,1);
\draw[->,dashed] (0.2,0) -- (1.3,2);
\draw[->,dashed] (1.7,2) -- (2.8,0);
\draw[->,dashed] (3.2,0) -- (4.7,0);
\draw[->,dashed] (4.7,0) -- (5.8,2);
\draw[->,dashed] (6.2,2) -- (7.3,0);
\draw[->,dashed] (7.7,0) -- (8.8,0);
  \draw[->] (0,-0.3) -- (1.3,-0.3) node[below,pos=0.5] {$f_{1}$};
  \draw[->] (1.5,-0.3) -- (2.8,-0.3) node[below,pos=0.5] {$f_{2}$};
  \draw[->] (3,-0.3) -- (4.3,-0.3) node[below,pos=0.5] {$f_{3}$};
  \draw[->] (4.5,-0.3) -- (5.8,-0.3) node[below,pos=0.5] {$f_{4}=f_{1}$};
  \draw[->] (6,-0.3) -- (7.3,-0.3) node[below,pos=0.5] {$f_{5}=f_{2}$};
  \draw[->] (7.5,-0.3) -- (8.8,-0.3) node[below,pos=0.5] {$f_{6}=f_{3}$};
\end{tikzpicture}
\end{center}
For this system we have  $\omega_{f_{1,\infty}}(1) = \{1,2,3\}$. 
But if we take $U=\{2\}$ and $V=\{3\}$ we get that 
$f_{1,n}(U) \cap V = \emptyset$  for all $n=0,1,2.\dots$ and therefore 
$f_{1,\infty}$ has no TT$_0$.
\end{proof}

\begin{exa}[TT $\not\Rightarrow\omega$, TT$_0 \not\Rightarrow \omega$ TT]\label{ExTT}
There are nonautonomous dynamical systems $f_{1,\infty}$ and $g_{1,\infty}$ on compact space $X$ 
such that \\ 
1. $f_{1,\infty}$ has TT but does not have $\omega$.\\
2. $g_{1,\infty}$ has TT$_0$ but does not have TT, and $\omega$.
\end{exa}
\begin{proof}
Take $X=\{0,1\}$ and $f_{1,\infty}$ on $X$ given by the following picture. 
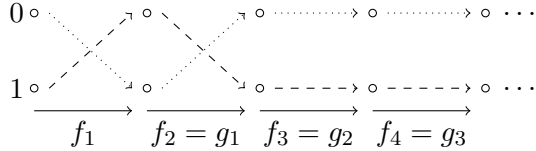
\begin{center}
\begin{tikzpicture}
  \draw (0,0) circle [radius=0.05] node[left] {$1$};
  \draw (0,1) circle [radius=0.05] node[left] {$0$};
  \draw (1.5,0) circle [radius=0.05] ;
  \draw (1.5,1) circle [radius=0.05] ;
  \draw (3,0) circle [radius=0.05] ;
  \draw (3,1) circle [radius=0.05] ;
  \draw (4.5,0) circle [radius=0.05] ;
  \draw (4.5,1) circle [radius=0.05] ;
  \draw (6,0) circle [radius=0.05] node[right] {$\ \ldots$};
  \draw (6,1) circle [radius=0.05] node[right] {$\ \ldots$};
  \draw[->] (0,-0.3) -- (1.3,-0.3) node[below,pos=0.5] {$f_{1}$};
  \draw[->] (1.5,-0.3) -- (2.8,-0.3) node[below,pos=0.5] {$f_{2}=g_{1}$};
  \draw[->] (3,-0.3) -- (4.3,-0.3) node[below,pos=0.5] {$f_{3}=g_{2}$};
  \draw[->] (4.5,-0.3) -- (5.8,-0.3) node[below,pos=0.5] {$f_{4}=g_{3}$};
  \draw[->, dotted] (0.2,1) -- (1.3, 0);
  \draw[->,dotted] (1.7, 0) -- (2.8,1);
  \draw[->,dotted] (3.2,1) -- (4.3,1);
  \draw[->,dotted] (4.7,1) -- (5.8,1);
  \draw[->, dashed] (0.2,0) -- (1.3, 1);
  \draw[->, dashed] (1.7, 1) -- (2.8,0);
  \draw[->, dashed] (3.2,0) -- (4.3,0);
  \draw[->, dashed] (4.7,0) -- (5.8,0);
\end{tikzpicture}
\end{center}
It is easy to see that $f_{1,\infty}$ has TT but there is no $x\in X$ 
with $\omega_{f_{1,\infty}}(x) = X$.

Let us put $g_{n} = f_{n+1}$ for $n = 1,2,\ldots$.
For $U = V = \{0\}$ there is no $n > 0$ such that 
$g_{1,n}(U) \cap V \neq\emptyset$, therefore 
$g_{1,\infty}$ does not have TT. Clearly $g_{1,\infty}$ has TT$_0$.
It is also easy to see that there is no point $x\in X$ with 
$\omega_{g_{1,\infty}}(x) = X$. 
\end{proof}

The proof of the following theorem is an extension and nonautonomous 
adaptation of the proof in \cite{Sil}.

\begin{thm}[TT$_0 \Leftrightarrow$ DDO]\label{TT0DDO}
Let $f_{1,\infty}$ be a nonautonomous dynamical system on 
a compact metric space $X$. Then $f_{1,\infty}$ has TT$_{0}$ 
if and only if it has DDO.
\end{thm}
\begin{proof}
Denote by $(B_{i})_{i=1}^\infty$ a countable basis in $X$, we may assume that 
$B_{i}\neq\emptyset$, for each i. Let us put 
$$
U_{i} = \bigcup_{n=0}^\infty f_{1,n}^{-1} (B_{i}). 
$$
Obviously each $U_{i}$ is open and nonempty (every $f_{1,n}$ is continuous). 
To show that it is a dense set consider $V \subset X$ open. We have $B_{i} \subset U_{i}$ 
because $f_{1,0} = \id_{X}$. Then $f_{1,n} (B_{i}) \subset U_{i}$, for any $n\in\N$.  
But by TT$_{0}$ for some $n \in \N\cup\{0\}$ the 
set $f_{1,n}^{-1}(B_{i})$ must intersect $V$. Therefore $U_{i}$ is dense.

Since $X$ is a Baire space and all $U_{i}$ are dense and open, 
$$
    D = \bigcap_{i=1}^\infty U_{i}
$$
is dense $G_\delta$. We show that every $x \in D$ has a dense orbit.
Let $W\subset X$ be an open set, there is a member of basis $B_{i} \subset W$. 
Because $x\in U_{i}$ there is an $n\in \N\cup\{0\}$ such that 
$f_{1,n}(x) \in B_{i} \subset W$. Therefore $\Orb_{f_{1,\infty}}(x)$ is dense.

Now suppose that there is a dense set $D$ whose every element has a dense orbit.
For a given pair $U,V$ of nonempty open set there is an $x \in\ U \cap D$. Since $x$ has 
a dense orbit there is an $n\in \N\cup\{0\}$ such that $f_{1,n}(x) \in V$ and 
we get $f_{1,n}(U) \cap V \neq \emptyset$.
\end{proof}

The following theorem summarizes the results obtained above and is 
the main result for the case of nonautonomous system on general compact metric space. 

\begin{thm}\label{MainComp}
Let $f_{1,\infty}$ be a nonautonomous system on a compact metric space $X$. 
Then the following condition 
\begin{itemize}
        \item[(D$\omega$)] There is a $G_\delta$-dense set $A\subset X$ such that $\omega_{f_{1,\infty}}(x) = X$ for any $x\in A$.
\end{itemize}
%
implies the following four equivalent conditions but converse implications does not hold
\begin{itemize}
    \item[(TT)] For every pair of nonempty open sets $U,V\subset X$, there is an $n\in\N$ 
        such that $f_{1,n}(U) \cap V \neq \emptyset$.
    \item[(TT$^{-}$)] For every nonempty open sets $U,V\subset X$, there is an $n\in\N$ 
        such that $f_{1,n}^{-1}(U) \cap V \neq \emptyset$.
    \item[(OP)] For every nonempty open $U\subset X$ the set $\bigcup_{n=1}^{\infty} f_{1,n}(U)$ is dense in $X$.  
    \item[(OP$^{-}$)] For every nonempty open $U\subset X$ the set $\bigcup_{n=1}^{\infty} f_{1,n}^{-1}(U)$ is dense in $X$.  
\end{itemize}
%
Each of them implies the following five equivalent conditions but converse implications does not hold
\begin{itemize}
    \item[(TT$_0$)] For every pair of nonempty open sets $U,V\subset X$, there is an $n\in\N\cup\{0\}$ 
        such that $f_{1,n}(U) \cap V \neq \emptyset$.
    \item[(TT$^{-}_{0}$)] For every nonempty open sets $U,V\subset X$, there is an $n\in\N\cup\{0\}$
        such that $f_{1,n}^{-1}(U) \cap V \neq \emptyset$.
    \item[(OP$_{0}$)] For every nonempty open $U\subset X$ the set $\bigcup_{n=0}^{\infty} f_{1,n}(U)$ is dense in $X$.  
    \item[(OP$^{-}_{0}$)] For every nonempty open $U\subset X$ the set $\bigcup_{n=0}^{\infty} f_{1,n}^{-1}(U)$ is dense in $X$.  
    \item[(DDO)] There is a $G_\delta$-dense set $A\subset X$ such that $\Orb_{f_{1,\infty}}(x)$ is dense  
        for any $x\in A$.
\end{itemize}
%
Each of them implies the following condition but converse implications does not hold
\begin{itemize}
    \item[(DO)] There is an $x \in X$ with a dense orbit.
\end{itemize}

\smallskip
Additionally, the condition D$\omega$ implies 
\begin{itemize}
    \item[($\omega$)] There is an $x \in X$ such that $\omega_{f_{1,\infty}}(x) = X$.
\end{itemize}
and this implies condition DO. But no other implication between $\omega$ and the properties above holds.
\end{thm}

Situation from Theorem~\ref{MainComp} can be described by the following scheme.
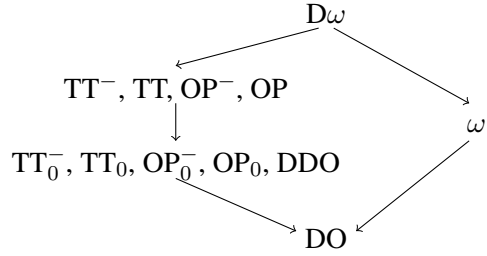
\begin{center}
\begin{tikzpicture}
    \draw(2,0) node {DO};
    \draw(0,1) node {TT$_0^-$, TT$_0$, OP$_0^-$, OP$_0$, DDO};
    \draw[->](1.9,2.8) -- (0,2.3);
    \draw(0,2) node {TT$^-$, TT, OP$^-$, OP};
    \draw[->](0,1.8) -- (0,1.3);
    \draw(2,3) node {D$\omega$};
    \draw[->](0,0.8) -- (1.6,0.1);
    \draw(4,1.5) node {$\omega$};
    \draw[->](3.9,1.3) -- (2.4,0.1);
    \draw[->] (2.1,2.8) -- (3.9,1.8);   
\end{tikzpicture}\\
General compact metric space
\end{center}

\begin{proof}
The equivalence of of TT, TT$^-$, OP and OP$^-$ follows from their definitions 
as well as the equivalence of TT$_0$, TT$_0^-$, OP$_0$ and OP$_0^-$. 

Equivalence of TT$_0$ and DDO holds by Theorem~\ref{TT0DDO}

Implication D$\omega \Rightarrow \omega$ and TT $\Rightarrow$ TT$_0$ hold trivially. 
Implication TT$_0 \Rightarrow $ DO follows by Theorem~\ref{TTODO}. 

Each point $x$ whose $\omega_{f_{1,\infty}}(x) = X$ must have a dense orbit and therefore $\omega \Rightarrow $~DO.
 
Now suppose D$\omega$. For arbitrary nonempty open $U,V$ there is an $x\in U$ with $\omega_{f_{1,\infty}}(x) = X$. 
Hence there is an $n \in \N$ such that $f_{1,n}(U) \cap V \neq\emptyset$. Thus we have TT.

Since DO does not imply $\omega$ or TT$_0$ in the case of autonomous system these 
implications does not hold in nonautonomous case. 

Example~\ref{ExOmega} shows that $\omega$ does not imply TT$_0$ and consequently TT and D$\omega$. 

On the other hand $g_{1,\infty}$ of Example~\ref{ExTT} shows that TT$_0$ does not imply TT and $\omega$ consequently also 
D$\omega$. The $f_{1,\infty}$ of the same example shows that also  TT does not imply $\omega$ and D$\omega$.
\end{proof}

\section{Spaces without an isolated point}

In the remainder, we assume that the compact metric space $X$ has no isolated point.

\begin{thm}[DDO $\Rightarrow$ D$\omega$]\label{DDODomega}
Let $f_{1,\infty}$ be a nonautonomous dynamical system on $X$ without an isolated 
point. Then DDO implies D$\omega$.
\end{thm}
\begin{proof}
Follows from the fact that in the compact space $X$ without an isolated point, 
any point $x\in X$ with a dense orbit has $\omega_{f_{1,\infty}}(x) = X$. 
To see this, for any $y \in X$  consider a sequence $(y_{n})_{n=1}^\infty$ 
of different points and also different from $y$ converging to $y$. 
Let $B_{n}$ be system of pairwise disjoint open balls with center at $y_{n}$ 
and radius smaller than $1/n$. 
Since $x$ has a dense orbit, for any $k\in\N$ there exist $n_{k} \in \N$ such 
that $x_{k} = f_{1,n_{k}}(x) \in B_{k}$. It is clear that $x_{k}$ converge to $y$.
\end{proof}

Many results concerning nonautonomous systems must add additional assumptions on 
the sequence $(f_{n})_{n=1}^\infty$ to be  valid. The most common such assumptions 
include uniform convergence and surjectivity of the mappings $f_{n}$. In the 
following example, which disproves the equivalence of TT and DO, we show that 
even by adding the mentioned assumptions DO does not imply TT.

\begin{exa}[DO $\not\Rightarrow$ TT]\label{ExTent}
There is a uniformly converging nonautonomous dynamical system of surjective 
maps on $[0,1]$ which is DO but not TT.
\end{exa}
\begin{proof}
Take $X=[0,1]$ and put
$$
f_{1}(x) = \left\{ 
\begin{array}{rl}
    0, & x \leq \frac{1}{4}, \\[3pt]
     4x - 1,& \frac14 < x \leq \frac12, \\[3pt]
     2 - 2x, & x > \frac12,
\end{array}
\right.
\hspace{7mm}
\begin{array}{c}
f_{n}(x) = T(x) =  \left\{ 
\begin{array}{rl}
     2x, & x \leq \frac{1}{2}, \\[3pt]
     2 - 2x, & x > \frac{1}{2},
\end{array}\right. \\[15pt]
n=2,3,\ldots
\end{array}
$$

\begin{center}
\begin{tikzpicture}[scale=3]
    \draw[thin] (0,0) rectangle (1,1);
    \draw[dotted] (0.5,0) -- (0.5,1);
    \draw[thick] (0,0) -- (0.25, 0) -- (0.5,1) -- (1,0); 
    \draw (0.78,0.6) node[above] {$f_1$};
    \draw (0,0) node[left] {$0$};
    \draw (0,1) node[left] {$1$};
    \draw (0,0) node[below] {$0$};
    \draw (1,0) node[below] {$1$};
    \draw (0.5,0) node[below] {$\frac12$};
    \draw (0.25,0) node[below] {$\frac14$};
\end{tikzpicture} \hspace{10mm}
\begin{tikzpicture}[scale=3]
    \draw[thin] (0,0) rectangle (1,1);
    \draw[dotted] (0.5,0) -- (0.5,1);
    \draw[thick] (0,0) -- (0.5,1) -- (1,0);
    \draw (0.78,0.6) node[above] {$f_n$};
    \draw (0,0) node[left] {$0$};
    \draw (0,1) node[left] {$1$};
    \draw (0,0) node[below] {$0$};
    \draw (1,0) node[below] {$1$};
    \draw (0.5,0) node[below] {$\frac12$};
\end{tikzpicture}
\end{center}

It is well know that tent-map $T(x)$ has a point $x_2 \in [0,1]$ with dense 
orbit $\{x_{2}, T(x_{2}), T^{2}(x_{2}),\ldots\} = \{x_{2}, f_{2}(x_{2}), f_{3}(f_{2}(x_{2})),\ldots\}$. 
Since $f_{1}$ is surjective there is $x_{1}$ for which $f_{1}(x_{1}) = x_{2}$ and therefore 
$\Orb_{f_{1,\infty}}(x_{1})$ is dense.

Now,  take  $U = (0,\frac14)$, then $U$ is mapped by $f_1$ to singleton $\{0\}$ and 
$0$ is a fixed point of all $f_{2},f_{3},\ldots$, it is clear that 
$f_{1,\infty}$ is not TT.
\end{proof}

The following theorem is the main result in this section. 

\begin{thm}\label{MainNoIso}
Let $f_{1,\infty}$ be a nonautonomous system on a compact metric space $X$ without 
an isolated point. 
Then the following ten conditions are equivalent.
\begin{itemize}
\item[(TT)] For every pair of nonempty open sets $U,V\subset X$, there is an $n\in\N$ 
        such that $f_{1,n}(U) \cap V \neq \emptyset$.
\item[(TT$_0$)] For every pair of nonempty open sets $U,V\subset X$, there is an $n\in\N\cup\{0\}$ 
        such that $f_{1,n}(U) \cap V \neq \emptyset$.
\item[(OP)] For every nonempty open $U\subset X$ the set $\bigcup_{n=1}^{\infty} f_{1,n}(U)$ is dense in $X$.
\item[(OP$_{0}$)] For every nonempty open $U\subset X$ the set $\bigcup_{n=0}^{\infty} f_{1,n}(U)$ is dense in $X$.
\item[(TT$^{-}$)] For every nonempty open sets $U,V\subset X$, there is an $n\in\N$ 
        such that $f_{1,n}^{-1}(U) \cap V \neq \emptyset$.
\item[(TT$_{0}^{-}$)] For every nonempty open sets $U,V\subset X$, there is an $n\in\N\cup\{0\}$
        such that $f_{1,n}^{-1}(U) \cap V \neq \emptyset$.
\item[(OP$^{-}$)] For every nonempty open $U\subset X$ the set $\bigcup_{n=1}^{\infty} f_{1,n}^{-1}(U)$ is dense in $X$.
\item[(OP$_{0}^{-}$)] For every nonempty open $U\subset X$ the set $\bigcup_{n=0}^{\infty} f_{1,n}^{-1}(U)$ is dense in $X$.
\item[(DDO)] There is a $G_\delta$-dense set $A\subset X$ such that $\Orb_{f_{1,\infty}}(x)$ is dense  
        for any $x\in A$.
\item[(D$\omega$)] There is a $G_\delta$-dense set $A\subset X$ such that $\omega_{f_{1,\infty}}(x) = X$ 
        for any $x\in A$.
\end{itemize}
 Each of them implies the following pair of equivalent conditions but converse implications does not hold
\begin{itemize}
\item[($\omega$)] There is an $x \in X$ such that $\omega_{f_{1,\infty}}(x) = X$.
\item[(DO)] There is an $x \in X$ with a dense orbit.
\end{itemize}
\end{thm}

Situation from Theorem~\ref{MainNoIso} can be described by the following scheme. 
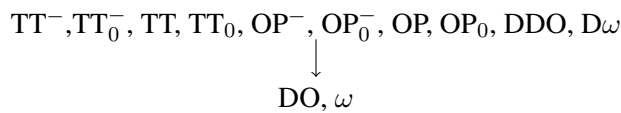
\begin{center}
\begin{tikzpicture}
\draw(0,1) node {TT$^-$,TT$_0^-$, TT, TT$_0$, OP$^-$, OP$_0^-$, OP, OP$_{0}$, DDO, D$\omega$};
\draw[->] (0,0.8) -- (0,0.3);
\draw (0,0) node {DO, $\omega$};
\end{tikzpicture}\\
Compact metric space without an isolated point
\end{center}

\begin{proof}
In view of more general case in Theorem~\ref{MainComp} the Theorem~\ref{DDODomega} concludes the 
equivalence of the ten properties

The implication from DO to $\omega$ follows again from the same argument as in the proof 
of Theorem~\ref{DDODomega}.

Finally, Example~\ref{ExTent} shows that DO does not imply TT.
\end{proof}

\section*{Acknowledgments}
The research was supported by RVO funding for I\v{C}47813059.

\end{document}